\documentclass[12pt,reqno]{amsart}
\newtheorem{theorem}{Theorem}[section]
\newtheorem{example}[theorem]{Example}
\newtheorem{examplen}[theorem]{Numerical Example}

\newtheorem{lemma}[theorem]{Lemma}
\newtheorem{proposition}[theorem]{Proposition}
\newtheorem{remark}[theorem]{Remark}

\usepackage{listings}
\usepackage{tikz}
\title[Heat equation with a cubic boundary]{A solution to the heat equation with a cubic  moving boundary}
\author{Gerardo Hern\'andez-del-Valle}
\address{Direcci\'on General de Investigaci\'on Econ\'omica, Banco de M\'exico, Av. 5 de Mayo \# 18, Col. Centro Hist\'orico, M\'exico, D.F. CP 06059}
\email{gerardo.hernandez@banxico.org.mx}
\begin{document}
\maketitle

\begin{abstract} In this  work we find a solution to the problem of the heat equation which is killed at a cubic boundary $f$. The solution turns out to be the convolution between the fundamental solution of the heat equation and a function $\phi$ which solves a third order ODE. However, the main contribution is the procedure itself, which links in a rather straightforward way, solutions of the heat equation with moving boundaries $f$ through the convolution of the heat kernel with suitable funtions $\phi$.
\end{abstract}
\section{Introduction}
Solutions to the heat equation in moving boundaries of the form
\begin{eqnarray}\label{heat}
\nu_t(t,x)=\frac{1}{2}\nu_{xx}(t,x),\nonumber\\ \nu(t,f(t))=0,\\
\nonumber (t,x)\in\mathbb{R}^+\times\mathbb{R},\quad f\in\mathbb{C}^2
\end{eqnarray}
appear prominently in applications [see for instance: Hernandez-del-Valle (2012) for applications on the first hitting problem of Brownian motion;  Bj\"ork (2009) in the valuation of barrier options; Davis and Pistorius (2010) in the quantification of counterparty risk; Martin-L\"of (1998) for  applications in biology]. In fact, explicit solutions to the problem described in (\ref{heat}) are well known in some particular cases. For instance:
\begin{enumerate}
\item[(a)] {\it Linear boundary\/.} For $b\in\mathbb{R}$, the following function
\begin{eqnarray}\label{linear1}
\nu(t,x)=\frac{x}{\sqrt{2\pi t^3}}\exp\left\{-\frac{x^2}{2t}\right\}+b\frac{1}{\sqrt{2\pi t}}\exp\left\{-\frac{x^2}{2t}\right\}.
\end{eqnarray}
solves (\ref{heat}) in the case in which $f(t)=-bt$. [See for instance Karatzas \& Shreve (1991) for an example of this function in the first hitting  time of Brownian motion to a linear boundary.]
\item[(b)] {\it Quadratic boundary\/.} Given that $\hbox{Ai}$ is an Airy function and $\xi\in\mathbb{R}^-$ is any of its roots, then 
\begin{eqnarray}\label{quadratic}
\nu(t,x)=\exp\left\{\frac{t^3}{12}+\frac{tx}{2}\right\}\hbox{Ai}\left(x+\frac{t^2}{4}\right)
\end{eqnarray}
is a solution of problem (\ref{heat}) when $f(t)=\xi-t^2/4$. See for instance Vall\'ee \& Soares (2004) for applications of the Airy function, including the solution to the heat equation.
\item[(c)] {\it Rayleigh type equation\/.} Let 
\begin{eqnarray*}
\nu(t,x)=\frac{1}{2\pi}\int_{-\infty}^{\infty}\exp\left\{i\lambda x-\frac{1}{2}\lambda^2t-\frac{\lambda^4}{4}\right\}d\lambda
\end{eqnarray*}
be the so-called Pearcey function. Then problem (\ref{heat}) is solved in the case in which the  function $f$ solves
\begin{eqnarray*}
f''(t)=2[f'(t)]^3-\frac{1}{2}tf'(t)-\frac{1}{4}f(t).
\end{eqnarray*} 
See Hern\'andez-del-Valle (2016) for its derivation and further related literature.
\item[(d)] In general one could try to find a function $f$ that solves (\ref{heat}), whenever $\nu$  is a  linear combination of $n$ solutions $\{\nu_j\}_{j=1,\dots,n}$ to the heat equation. That is
\begin{eqnarray*}
\nu(t,x)=a_1\nu_1(t,x)+\cdots+a_n\nu_n(t,x).
\end{eqnarray*}
See for instance (\ref{linear1}). 
\end{enumerate}
We note that there is no straightforward way to solve problem (\ref{heat})  [see for instance De Lillo \& Fokas (2007) which study this problem using integral equations]. Let us distinguish the following possibilities: 
\begin{enumerate}
\item[(a.1)] Given that $\nu$ solves the heat equation then $f$ solves (\ref{heat}).
\item[(b.1)] For a given moving boundary $f$, there exists a solution $\nu$ to the heat equation, such that (\ref{heat}) holds.
\end{enumerate}
The main objective of this work is to find a solution of the heat equation which solves problem (\ref{heat}) in the case in which the moving boundary $f$ is cubic, that is, $f(t)=bt^3$, for  $b\in\mathbb{R}$ and $t\geq 0$. However, as will show throughout the document,  the methodology could in principle be used to solve this problem  to at least the case in which the boundaries $f$ are polynomial in $t$.

The technique used to achieve our goal is remarkably straightforward, and is based in  analyzing  the convolution between the {\it fundamental\/} solution of the  heat equation and some real valued and sufficiently smooth function $\phi$. In Hern\'andez-del-Valle (2016) the author uses similar arguments, but in that work the problem is  of the type (a.1). In contrast in the present paper, the problem is as the one posed in (b.1). What we will show is that making use of the technique described within, we may find a function $\phi$ which convolved with the fundamental solution to the heat equations leads to solutions of the type (b.1) in the case in which the boundary is quadratic and cubic. We suspect that the technique could be generalized to the case in which the boundary has an arbitrary integer power.  


The paper is organized as follows, in Section \ref{sec2} the technique used to link solutions $v$ of heat equation  with  moving boundaries $f$ is introduced in the case in which the linking function $\phi$ is $\mathbb{C}^2$.  Subsequently, in Section \ref{sec3}, we derive the solution of the heat equation with a cubic absorbing boundary in detail. In this section we also provide a numerical example which can be replicated. Finally, we conclude in Section \ref{sec4} with some final remarks and work in progress.

\section{Preliminary results}\label{sec2}
\begin{remark} For the remainder of this document, given any function $f$, its $n$-th partial derivative with respect to the state variable $x$ will be denoted as $f^{(n)}$.
\end{remark}
In this section we derive some algebraic properties, which  subsequently will be extended in the following section, of the convolution between the fundamental solution of the heat equation and a $\mathbb{C}^2$ function $\phi$ . To this end we will make use of the following Lemma
\begin{lemma} Let $\{v_j\}_{n\geq 0}$ be the sequence of functions  defined as
\begin{eqnarray}\label{exp}
e^{i\lambda x-\lambda^2t/2}=\sum\limits_{j=0}^\infty v_j(t/2,x)\frac{(i\lambda)^j}{j!},
\end{eqnarray}
be each a solution to the heat equation. Then
\begin{eqnarray}\label{deriv}
\frac{d^p}{d\lambda^p}\left[e^{i\lambda x-\lambda^2t/2}\right]=v_p(-t/2,ix-\lambda t)e^{i\lambda x-\lambda^2t/2}.
\end{eqnarray}

\end{lemma}
\begin{proof}
For $j\geq 1$ the following recurrence relation (which can be verified by induction) holds for  the $v_j$  in (\ref{exp}): 
\begin{eqnarray*}
v_{j+1}(t,x)=xv_j(t,x)+jtv_{j-1}(t,x).
\end{eqnarray*}
Thus
\begin{eqnarray*}
\frac{d}{d\lambda}\left[e^{i\lambda x-\lambda^2t/2}\right]&=&\sum\limits_{j=0}^\infty v_{j+1}(t/2,x)i^{j+1}\frac{\lambda^j}{j!}\\
&=&i\sum\limits_{j=0}^\infty v_{j+1}(t/2,x)i^{j}\frac{\lambda^j}{j!}\\
&=&i\sum\limits_{j=0}^\infty(xv_j+jtv_{j-1})i^{j}\frac{\lambda^j}{j!}\\
&=&ix\sum\limits_{j=0}^\infty v_j i^{j}\frac{\lambda^j}{j!}+i\sum\limits_{j=0}^\infty jtv_{j-1}i^{j}\frac{\lambda^j}{j!}\\
&=&ix\sum\limits_{j=0}^\infty v_j i^{j}\frac{\lambda^j}{j!}+i^2\lambda t\sum\limits_{j=0}^\infty v_{j-1}i^{j-1}\frac{\lambda^{j-1}}{(j-1)!}\\
&=&(ix-\lambda t)e^{i\lambda x-\lambda^2t/2}\\
&=&v_1(t,ix-\lambda t)e^{i\lambda x-\lambda^2t/2}.
\end{eqnarray*}
So if we let
\begin{eqnarray*}
\mathrm{h}(\lambda)=e^{i\lambda x-\lambda^2t/2}\quad\hbox{and}\quad y=ix-\lambda t
\end{eqnarray*}
then
\begin{eqnarray*}
\frac{d^2\mathrm{h}}{d\lambda^2}&=&(-tv_0+yv_1)\mathrm{h}\\
&=&v_2(-t/2,y)\mathrm{h}
\end{eqnarray*}
and
\begin{eqnarray*}
\frac{d^3\mathrm{h}}{d\lambda^3}&=&\frac{\partial}{\partial\lambda}v_2\mathrm{h}+v_2\frac{\partial}{\partial \lambda}\mathrm{h}\\
&=&\frac{\partial y}{\partial\lambda}\frac{\partial v_2}{\partial y}\mathrm{h}+v_2\frac{\partial\mathrm{h}}{\partial\lambda}\\
&=&-2tv_1\mathrm{h}+yv_2\mathrm{h}\\
&=&(yv_2-2tv_1)\mathrm{h}\\
&=&v_3(-t/2,y)\mathrm{h}.
\end{eqnarray*}
Hence an induction argument shows that  (\ref{deriv}) holds.
\end{proof}

\begin{proposition}\label{prop1}
For $p, q, r\in\mathbb{N}$, let $\phi$ be a  solution to 
\begin{eqnarray}\label{phi}
x^p\phi^{(2)}(x)&=&ax^q\phi^{(1)}(x)+bx^r\phi^{(0)}(x) \qquad x\in\mathbb{R},
\end{eqnarray}
and  let $v_n$ be as in (\ref{exp}).  
Then given that $\overline{\phi}$ is the Fourier transform of a solution $\phi$ to (\ref{phi}), there exists a function $\nu$ that solves the heat equation $\nu$ defined as
\begin{eqnarray*}
\nu(t,x):=\frac{1}{2\pi}\int_{-\infty}^{\infty}\overline{\phi}(\lambda)e^{i\lambda x-\lambda^2t/2}d\lambda\qquad (t,x)\in\mathbb{R}^+\times\mathbb{R},
\end{eqnarray*}
and satisfies the following relationship as well
\begin{eqnarray}\label{th}
&&(-i)^p\int (i\lambda)^2\overline{\phi}(\lambda) v_p(-t/2,ix-\lambda t) e^{i\lambda x-\lambda^2t/2}d\lambda\\
\nonumber&&\qquad\qquad =(-i)^qa\int (i\lambda)^1\overline{\phi}(\lambda) v_q(-t/2,ix-\lambda t)  e^{i\lambda x-\lambda^2t/2}d\lambda\\
\nonumber&&\qquad\quad\qquad+(-i)^rb\int \overline{\phi}(\lambda) v_r(-t/2,ix-\lambda t) e^{i\lambda x-\lambda^2t/2}d\lambda.
\end{eqnarray}
\end{proposition}
\begin{proof}
Let $p, q, r\in\mathbb{N}$, and let $\phi$ be a  solution to (\ref{phi}).
Applying the Fourier transform to  both sides of identity (\ref{phi}) yields  
\begin{eqnarray*}
i^p\frac{d^p}{d\lambda^p}\left[(i\lambda)^2\overline{\phi}(\lambda)\right]&=&ai^q\frac{d^q}{d\lambda^q}\left[(i\lambda)^1\overline{\phi}(\lambda)\right]+bi^r\frac{d^r}{d\lambda^r}\left[\overline{\phi}(\lambda)\right].
\end{eqnarray*}
Next, we convolve (or average) each term of the previous expresion, with the Fourier transform of the fundamental solution to the heat equation to obtain
\begin{eqnarray}\label{con}
&&\int e^{i\lambda x-\lambda^2t/2}i^p\frac{d^p}{d\lambda^p}\left[(i\lambda)^2\overline{\phi}\right]d\lambda\\
\nonumber&&\qquad=a\int e^{i\lambda x-\lambda^2t/2}i^q\frac{d^q}{d\lambda^q}\left[(i\lambda)^1\overline{\phi}\right]d\lambda +b\int e^{i\lambda x-\lambda^2t/2} i^r\frac{d^r}{d\lambda^r}\left[\overline{\phi}\right]d\lambda.
\end{eqnarray}
In order to make the previous expression depend only on the derivatives of 
$$
e^{i\lambda x-\lambda^2t/2},
$$
we recall that for arbitrary differentiable functions $f$ and $g$, the integration by parts formula reads
\begin{eqnarray*}
\int f\frac{d^pg}{dx^p}dx&=&\int f d\left[\frac{d^{p-1}g}{dx^{p-1}}\right]\\
\nonumber&=& f\left[\frac{d^{p-1}g}{dx^{p-1}}\right]-\int \left[\frac{d^{p-1}g}{dx^{p-1}}\right]f^{(1)}dx\\
\nonumber&=&-\int \left[\frac{d^{p-1}g}{dx^{p-1}}\right]f^{(1)}dx\\
\nonumber&=&\int  \left[\frac{d^{p-2}g}{dx^{p-2}}\right]f^{(2)}dx\\
\nonumber&\vdots&\\
\nonumber&=&(-1)^p\int gf^{(p)}dx
\end{eqnarray*}
under appropriate conditions either on the functions $f$ and $g$ or on the limits of integration. 
Hence, the identity in (\ref{con}) can also be expressed as
\begin{eqnarray*}
&&(-i)^p\int (i\lambda)^2\overline{\phi}\frac{d^p}{d\lambda^p}\left[ e^{i\lambda x-\lambda^2t/2}\right]d\lambda\\
&&\qquad\qquad =(-i)^qa\int (i\lambda)^1\overline{\phi}\frac{d^q}{d\lambda^q}\left[ e^{i\lambda x-\lambda^2t/2}\right]d\lambda\\
&&\qquad\quad\qquad+(-i)^rb\int (i\lambda)^0\overline{\phi}\frac{d^r}{d\lambda^r}\left[ e^{i\lambda x-\lambda^2t/2}\right]d\lambda.
\end{eqnarray*}
Thus, equation (\ref{con}) follows from the latter equality and (\ref{deriv}).
\end{proof}
Furthermore, we observe the following:
\begin{remark}\label{rem1} Suppose there exist a pair of functions $\nu$ and $f$ that solve the moving boundary problem of the heat equation (\ref{heat}). Then, from the definition of the heat equation, the following identities should hold
\begin{eqnarray*}
f'(t)\nu^{(1)}(t,f(t))+\frac{1}{2}\nu^{(2)}(t,f(t))=0
\end{eqnarray*}
and
\begin{eqnarray*}
f''(t)\nu^{(1)}+f'(t)(f'(t)\nu^{(2)}+\nu^{(3)})+\frac{1}{4}\nu^{(4)}=0.
\end{eqnarray*}
\end{remark}
We are now ready to present the main result of this section which will make use of  Proposition \ref{prop1} and  Remark \ref{rem1}. Furthermore, the result links a family of paired functions $\nu$ and $f$, which solve (\ref{heat}), through a specific $\mathbb{C}^2$ function $\phi$.
\begin{theorem}\label{th2} For given fixed coefficients $d_0,d_1,c_0,c_1,c_2\in\mathbb{R}$, let $\phi$ be a real-valued solution of the following ODE
\begin{eqnarray}\label{ths2}
\phi^{(2)}(x)&=& \sum\limits_{j=0}^1d_jx^j\phi^{(1)}(x)+\sum\limits_{j=0}^{2}c_jx^j\phi(x), \qquad x\in\mathbb{R}.
\end{eqnarray}
In addition  at least one of the coefficients $d_1$, $c_2$ are different from zero. Then, if  there exists a function $f$ for  which (\ref{heat}) holds it should be of the following form
\begin{eqnarray*}
f(t)=\frac{-d_0d_1-2c_1-2c_2d_0t+c_1d_1t}{d_1^2+4c_2}+\sqrt{-1+d_1t+c_2t^2}\cdot \mathcal{C},
\end{eqnarray*}
where $\mathcal{C}$ is an arbitrary constant.
\end{theorem}

\begin{proof}
If the function $\phi$ is a solution of  (\ref{ths2}), it follows from Proposition \ref{prop1} that its convolution with the fundamental solution of the heat equation yields
\begin{eqnarray}\label{id1}
&&\nonumber(1-d_1t-c_2t^2)\nu^{(2)}(t,x)\\
&&\qquad=(d_0+d_1x+c_1t+c_2 2tx)\nu^{(1)}(t,x)\\
&&\nonumber\qquad\quad+(c_0+c_1x+c_2x^2+c_2t/2)\nu^{(0)}(t,x).
\end{eqnarray}
Now, suppose that there exists a function $f$ such that $\nu(t,f(t))=0$. Then, from  (\ref{id1}),
\begin{eqnarray*}
&&(1-d_1t-c_2t^2)\nu^{(2)}(t,f(t))\\
&&\qquad\qquad=(d_0+c_1t+[d_1+2c_2t]f(t))\nu^{(1)}(t,f(t))
\end{eqnarray*}
and, from Remark \ref{rem1},
\begin{eqnarray*}
\nu^{(2)}(t,f(t))=-2f'(t)\nu^{(1)}(t,f(t)).
\end{eqnarray*}
Thus, equating the two previous expressions we have that
\begin{eqnarray*}
-2f'(t)(1-d_1t-c_2t^2)=(d_0+c_1t+[d_1+2c_2t]f(t)).
\end{eqnarray*}
This implies that $f$, which solves the later ODE, has the following general solution (by standard techniques) as long as at least one of the coefficients $d_1$, $c_2$ are different from zero
\begin{eqnarray*}
f(t)=\frac{-d_0d_1-2c_1-2c_2d_0t+c_1d_1t}{d_1^2+4c_2}+\sqrt{-1+d_1t+c_2t^2}\cdot \mathcal{C}.
\end{eqnarray*}
\end{proof}
\section{Derivation of the cubic boundary}\label{sec3}
In this section we present the main result of this paper. Namely, given that $f$ is a cubic moving boundary, we find a $\mathbb{C}^3$ function $\phi$ which solves the moving boundary problem of the heat equation (\ref{heat}). 
\begin{theorem}\label{th3}  Suppose that the moving boundary $f$ in (\ref{heat}) is $f(t)=-b_2/8\cdot t^3$. Furthermore let 
 $b_2\in\mathbb{R}\backslash \{0\}$, and $\phi$ be a real-valued function that satisfies that
\begin{eqnarray}\label{ph1}
\phi'''(x)=b_2x^2\phi(x).
\end{eqnarray}
Then there exists a real valued solution of (\ref{ph1}), which  convoluted with the heat kernel yields a function $\nu$ that solves problem (\ref{heat}).

\end{theorem}
Before we present  the proof of the theorem, we provide an example.

\begin{example}\label{ex3} For $b_2=-1$,  the function  
\begin{eqnarray*}
\phi(x):=\frac{x}{5^{3/5}}\cdot\!\!\!\! \phantom{1}_{0}F_{2}\left[\Big{\{}\Big{\}},\Big{\{}\frac{4}{5},\frac{6}{5}\Big{\}},-\frac{x^5}{125}\right],
\end{eqnarray*}
defined in terms of the generalized hypergeometric function $\cdot\!\!\!\! \phantom{1}_{0}F_{2}$, solves (\ref{ph1}). Furthermore, let
\begin{eqnarray*}
h(t,x):=\frac{1}{\sqrt{2\pi t}}\exp\left\{-\frac{x^2}{2t}\right\}\qquad\hbox{for } x\in\mathbb{R},\enskip t\geq 0.
\end{eqnarray*}
Then, the convolution of $h$ and $\nu$
\begin{eqnarray*}
\nu(t,y)=\int_{-\infty}^\infty h(t,x)\phi(y-x)dx
\end{eqnarray*}
is a solution to the heat equation in problem (\ref{heat}) when $f(t)=t^3/8$. That is
\begin{eqnarray*}
\nu(t,t^3/8)=0\qquad \forall t\geq 0.
\end{eqnarray*}
\end{example}
\begin{examplen} We carry out some numerical calculations with the functions described in Example \ref{ex3} using Mathematica. The code and results are the following:\\ 
{\tiny
\begin{lstlisting}
DSolve[{f'''[x] == -x^2*f[x], f[0] == 0}, f[x], x]
\end{lstlisting}
\begin{lstlisting}
f[x_, C_, D_] := 
 1/25 (5 5^(2/5)x*C*HypergeometricPFQ[{}, {4/5, 6/5}, -(x^5/125)] + 
    5^(4/5)x^2*D*HypergeometricPFQ[{}, {6/5, 7/5}, -(x^5/125)])
\end{lstlisting}
\begin{lstlisting}
gau[t_, x_, y_] := 1/Sqrt[2*Pi*t]*Exp[-(x - y)^2/(2*t)]
\end{lstlisting}
\begin{lstlisting}
Table[NIntegrate[
  f[-x + (i/20)^3/8, 1, 0]*gau[i/20, x, 0], {x, -Infinity, 
   Infinity}], {i, 10}]
Out={-5.96311*10^-19, -6.83047*10^-18, -5.64802*10^-18, 
 9.97466*10^-18, -3.25261*10^-18, -2.71051*10^-18, -7.80626*10^-18, 
-3.46945*10^-18, -6.50521*10^-18, -9.86624*10^-18}
\end{lstlisting}
\begin{lstlisting}
Table[NIntegrate[
  f[-x + (i/20)^3/8 + 2, 1, 0]*gau[i/20, x, 0], {x, -Infinity, 
   Infinity}], {i, 10}]
Out={0.530853, 0.49626, 0.46151, 0.426968, 0.392976, 0.359844, 0.327851, 
0.297236, 0.268196, 0.240891}
\end{lstlisting}
}
\end{examplen}
The proof of Theorem \ref{th3} is in the spirit of the proof of Theorem \ref{th2}. The only difference is that that the function $\phi$, which links $\nu$ and the boundary $f$ in problem (\ref{heat}), is now $\mathbb{C}^3$ instead of $\mathbb{C}^2$.

\begin{proof}[Proof of Theorem \ref{th3}] For arbitrary constants $d_0,d_1,c_0,c_1,c_2,b_0,b_1,b_2,b_3\in\mathbb{R}$ let $\phi$ be a real-valued solution of the following equation
\begin{eqnarray*}
\phi^{(3)}=\sum\limits_{j=0}^1d_jx^j\phi^{(2)}+\sum\limits_{j=0}^2c_jx^j\phi^{(1)}+\sum\limits_{j=0}^3b_jx^j\phi^{(0)}.
\end{eqnarray*}
A direct application of Proposition \ref{prop1} to the latter identity yields
\begin{eqnarray*}
\nu^{(3)}&=&d_0\nu^{(2)}+d_1x\nu^{(2)}+d_1t\nu^{(3)}+c_0\nu^{(1)}+c_1x\nu^{(1)}+c_1t\nu^{(2)}\\
&&+c_2x^2\nu^{(1)}+c_22tx\nu^{(2)}+c_2t^2\nu^{(3)}+c_2t\nu^{(1)}\\
&&+b_0\nu^{(0)}+b_1x\nu^{(0)}+b_1t\nu^{(1)}+b_2x^2\nu^{(0)}+b_22tx\nu^{(1)}\\
&&+b_2t^2\nu^{(2)}+b_2t\nu^{(0)}+b_33t^2\nu^{(1)}+b_3t^3\nu^{(3)}\\
&&+b_33tx\nu^{(0)}+b_33t^2x\nu^{(2)}+b_33tx^2\nu^{(1)}+b_3x^3\nu^{(0)}.
\end{eqnarray*}
After factorizing coefficients, taking derivatives with respect to $x$, and evaluating at $x=f(t)$, yields
\begin{eqnarray*}
&&(1-d_1t-c_2t^2-b_3t^3)\nu^{(3)}\\
&&\qquad=(d_0+d_1x+c_1t+c_22tx+b_2t^2+b_33t^2x)\nu^{(2)}\\
&&\qquad\quad+\big{(}c_0+c_1x+c_2x^2+c_2t+b_1t+b_22tx+b_33t^2+b_33tx^2\big{)}\nu^{(1)}
\end{eqnarray*}
and
\begin{eqnarray*}
&&(1-d_1t-c_2t^2-b_3t^3)\nu^{(4)}\\
&&\qquad =(d_0+d_1x+c_1t+c_22tx+b_2t^2+b_33t^2x)\nu^{(3)}\\
&&\qquad\quad+\big{(}d_1+c_0+c_1x+c_2x^2+c_22t+c_2t+b_1t+b_22tx\\
&&\qquad\quad\enskip+b_33t+b_33t^2+b_33tx^2\big{)}\nu^{(2)}\\
&&\qquad\quad+\big{(}c_1+c_22x+b_0+b_1x+b_2x^2+b_22t+b_2t\\
&&\qquad\quad\enskip+b_33tx+b_36tx+b_3x^3\big{)}\nu^{(1)}.
\end{eqnarray*}
Next, if
\begin{eqnarray*}
d_1=c_2=b_3=0
\end{eqnarray*}
it follows that
\begin{eqnarray}\label{cu1}
\nu^{(3)}&=&(d_0+c_1t+b_2t^2)\nu^{(2)}+(c_0+c_1x+b_1t+b_22tx)\nu^{(1)}\\
\nu^{(4)}&=&(d_0+c_1t+b_2t^2)\nu^{(3)}\label{cu2}\\
&&+(c_0+c_1x+b_1t+b_22tx)\nu^{(2)}\nonumber\\
&&+(c_1+b_0+b_1x+b_2x^2+b_22t+b_2t)\nu^{(1)}.\nonumber
\end{eqnarray}
From (\ref{cu1}) and the first identity in  Remark \ref{rem1} we obtain
\begin{eqnarray*}
\nu^{(3)}&=&\Big{[}-2f\rq{}(t)\{d_0+c_1t+b_2t^2\}\\
&&+(c_0+c_1f(t)+b_1t+b_22tf(t))\Big{]}\nu^{(1)}.
\end{eqnarray*}
Next, from the second equation of Remark \ref{rem1} and the previous identity
\begin{eqnarray}\label{nu1}
\nu^{(4)}&=&[-4f\rq{}\rq{}(t)+8(f\rq{}(t))^3]\nu^{(1)}\\
\nonumber&&-4f\rq{}(t)\big{[}-2f\rq{}(t)\{d_0+c_1t+b_2t^2\}\\
\nonumber&&\enskip+(c_0+c_1f(t)+b_1t+b_22tf(t))\big{]}\nu^{(1)}.
\end{eqnarray}
In turn, from (\ref{cu2})
\begin{eqnarray}\label{nu2}
\nu^{(4)}&=&(d_0+c_1t+b_2t^2)\big{[}-2f\rq{}(t)\{d_0+c_1t+b_2t^2\}\\
\nonumber&&\enskip+(c_0+c_1f(t)+b_1t+b_22tf(t))\big{]}\nu^{(1)}\\
\nonumber&&+(c_0+c_1f(t)+b_1t+b_22tf(t))(-2f\rq{}(t))\nu^{(1)}\\
\nonumber&&+(c_1+b_0+b_1f(t)+b_2f^2(t)+b_23t)\nu^{(1)}.
\end{eqnarray}
Thus equating the right-hand sides of equations (\ref{nu1}) and (\ref{nu2}) it follows that
\begin{eqnarray*}
&&-4f\rq{}\rq{}(t)+8(f\rq{}(t))^3+8(f\rq{}(t))^2\{d_0+c_1t+b_2t^2\}\\
&&\enskip-4f\rq{}(t)\{c_0+c_1f(t)+b_1t+b_22tf(t)\}\\
&&\quad=-2f\rq{}(t)\{d_0+c_1t+b_2t^2\}^2\\
&&\qquad\quad+(c_0+c_1f(t)+b_1t+b_22tf(t))(d_0+c_1t+b_2t^2)\\
&&\qquad\quad-2f\rq{}(t)(c_0+c_1f(t)+b_1t+b_22tf(t))\\
&&\qquad\quad+(c_1+b_0+b_1f(t)+b_2f^2(t)+b_23t),
\end{eqnarray*}
which is further simplified by setting $d_0=c_1=b_1=c_0=0$
\begin{eqnarray*}
&&-4f\rq{}\rq{}(t)+8(f\rq{}(t))^3+8(f\rq{}(t))^2b_2t^2-4f\rq{}(t)2tb_2f(t)\\
&&\qquad=-2f\rq{}(t)b^2_2t^4+2tb_2f(t)b_2t^2\\
&&\quad\quad -2f\rq{}(t)2b_2tf(t)+b_2f^2(t)+3b_2t,
\end{eqnarray*}
or equivalently
\begin{eqnarray*}
&&-4f\rq{}\rq{}(t)+2f\rq{}(t)[4(f\rq{}(t))^2+4b_2t^2f\rq{}(t)-2b_2tf(t)+b_2^2t^4]\\
&&\qquad-b_2f(t)[2b_2t^3+f(t)]-3b_2t=0.
\end{eqnarray*}
Now, in order the verify the statement of the theorem, let $f(t)=\delta t^3$, so $f\rq{}(t)=3\delta t^2$, $f\rq{}\rq{}(t)=6\delta t$. Then substitute the values of $f$, $f'$ and $f''$ in the latter identity, in order to obtain
\begin{eqnarray*}
&&-24\delta t-3b_2t+2(3\delta t^2)(4\cdot 9\delta^2t^4+4b_2t^2\cdot 3\delta t^2-2b_2t\delta t^3+b_2^2t^4)\\
&&\qquad -b_2\delta t^3(2b_2t^3+\delta t^3)=0,
\end{eqnarray*}
or equivalently
\begin{eqnarray*}
&&-3t(8\delta+b_2)+6\delta t^2(36\delta^2t^4+12b_2\delta t^4-2b_2\delta t^4+b_2^2t^4)\\
&&\qquad -b_2\delta t^3(2b_2t^3+\delta t^3)=0.
\end{eqnarray*}
Factorizing in terms of $t$ and $t^6$ we have
\begin{eqnarray*}
&&-3t(8\delta+b_2)\\
&&\qquad \delta t^6(216\delta^2+72 b_2\delta-12b_2\delta+6b_2^2-2b_2^2-b_2\delta)=0.
\end{eqnarray*}
Finally for the last expression to hold for all $t\geq 0$ it should hold that
\begin{eqnarray*}
\delta=-b_2/8.
\end{eqnarray*}
But this also yields
\begin{eqnarray*}
216 b_2^2/64-59b_2^2/8+4b^2/2=0
\end{eqnarray*}
and thus the proof is complete.
\end{proof}

\section{Concluding Remarks}\label{sec4}
In this work we find a solution to the problem of the heat equation equation with a cubic killing moving boundary. The procedure described within also suggests there might be a relationship between the smoothness of the convoluting function $\phi$ and the power of the moving boundary $f$. In particular, we show in the present work that for $p=2,3$ this relationship holds. A more general descripction of the procedure presented in the present work is in progress.

\end{document}